\documentclass[twoside,12pt,leqno]{article}
\advance\topmargin by -\headheight\advance\topmargin by -\headsep
\newcommand{\tmvolxx}{18}
\newcommand{\tmyearyyyy}{2008}
\newcommand{\FirstPageHead}[3]{
{\footnotesize
\vskip -8mm
\centerline{\href{http://wwwen.uni.lu/content/download/20665/255970/file/171411-UNICL-LIVRE_INT_WEB.pdf}%
{Travaux math\'ematiques}, \quad Volume #1 (#2), #3,\quad \copyright\  Universit\'e du Luxembourg}}
\vspace{-3mm}}
\usepackage{amsmath}
\usepackage{amsthm}
\usepackage{amssymb}
\usepackage{amscd}
\usepackage{graphicx}
\usepackage{epsfig}
\usepackage[pdfstartview=FitH,pdfnewwindow=true,pdftitle={Variational problem in concircular geometry},pdfdisplaydoctitle=true,colorlinks,pagebackref,bookmarksopen,breaklinks]{hyperref}
 \newtheorem{theorem}{Theorem}[section]
 \newtheorem{lemma}{Lemma}[section]
 \newtheorem{proposition}{Proposition}[section]
 
 \newtheorem{claim}{Claim}[section]
 
 \theoremstyle{definition}

 \newtheorem*{remarku}{Remark}
 
\allowdisplaybreaks
\begin{document}
\thispagestyle{empty}
\FirstPageHead{\tmvolxx}{\tmyearyyyy}{125--137}
\newtheorem*{remarksu}{Remarks}
\newcommand{\1}{{\scriptscriptstyle1}}
\newcommand{\D}{{\mathbf D_{\boldsymbol1}}}
\newcommand{\T}{{\tilde D}}
\newcommand{\bcdot}{{{\cdot}}}
\newcommand{\bx}{{\boldsymbol\chi}}
\newcommand{\bp}{{\boldsymbol\partial}}
\newcommand{\bu}{{\mathbf u}}
\newcommand{\bw}{{\mathbf u^{\boldsymbol\prime}}}
\newcommand{\du}{{\mathbf{\dot u}}}
\newcommand{\ddu}{{\mathbf{\ddot u}}}
\newcommand{\pr}[2]{{\mathbf{\left(#1\cdot#2\right)}}}
\newcommand{\pw}[2]{{\mathbf{\left(#1\wedge#2\right)}}}
\newcommand{\nw}[2]{{\mathbf{\|#1\wedge#2\|}}}
\newcommand{\bnu}{{\|\mathbf u\|}}
\newcommand{\pnu}[1]{{\|\mathbf u\|^{#1}}}
\newcommand{\msp}{{\mspace{1mu}}}
\newcommand{\ep}[2]{{\epsilon_{#1#2}}}
\newcommand{\dg}{{\sqrt{|\det[g_{ij}]|}\;}}
\newcommand{\bpr}{{{\boldsymbol'}}}
\newcommand{\by}{{{\boldsymbol\upsilon}}}
\newcommand{\bs}{{{\boldsymbol\sigma}}}
\newcommand{\bt}{{{\boldsymbol\theta}}}
\newcommand{\R}{{{\boldsymbol{\mathcal R}}}}
\markboth{Roman Matsyuk}{Variational problem in concircular geometry}
\bigskip
\bigskip
\begin{center}
\renewcommand{\thefootnote}{\fnsymbol{footnote}}
{\Large Second order variational problem and \\
2--dimensional concircular geometry}\footnote[1]{Received: November 20, 2007}\footnote[2]
{Research supported by the grant GA\v CR 201/06/0922 of the Czech Science Foundation}
 \end{center}
\bigskip
\centerline{{\large by  Roman Matsyuk}}
\vspace*{.7cm}
\begin{abstract}
It is proved that the set of geodesic circles in two dimensions may be given a variational
description and the explicit form of it is presented. In the limit case of
the Euclidean geometry a certain claim of uniqueness of such description
is proved. A formal notion of `spin' force is discovered as a by-product
of the variation procedure involving the acceleration.
\end{abstract}
\pagestyle{myheadings}
\section{Introduction}
The concircular geometry deals with geodesic circles in (pseudo)-Riemannian
space.
Geodesic circles in two dimensions are those curves in 2-dimensional
(pseudo)-Riemannian space who preserve the Frenet curvature along them. In
relativity theory this coincides with the definition of the uniformly
accelerated one-dimensional motion of a test particle. The ordinary
differential equation to govern such curves has order
three~\cite{matsyuk:b1}. Thus the Lagrange function should
involve second derivatives and, at the same time, it should depend linearly
on them.

Aiming at simplification of the exposition and of the accompanying notations, let us
agree not to be confused with such notions as vector or bivector norm
in pseudo-Riemannian geometry.
Although the outcome of present investigation lucidly concerns both the proper
Riemannian and the pseudo-Riemannian geometries, for the sake of prudence
one may restrict oneself to the case of proper Riemannian space, and it still
will remain evident, wherein the results will be valid in actually
 the pseudo-Riemannian framework as well. Thus hereinafter we shall somewhat
 vaguely use the terms {\it Riemannian\/} and {\it Euclidean,} keeping in mind
 that strictly speaking, some details of pure technical developments can in fact
 apply only to proper Riemannian case.

Consider the following Lagrange function in 2--dimensional
Euclidean space:
\begin{equation}\label{matsyuk:1}
L = L_{II} + L_{I} = {\frac{{\ep i j u^{i}\dot {u}^{j}}}{{{\left\| {{\rm {\bf
u}}} \right\|}^{3}}}} - m{\left\| {{\rm {\bf u}}} \right\|}\;,
\end{equation}
\noindent
with $\ep i j $ denoting the skew-symmetric covariant Levi--Civita symbol. The first
addend, $L_{II} $, is the so-called signed first Frenet curvature of a path.
Further in this contribution we show that the
expression~(\ref{matsyuk:1}) as a candidate for the Lagrange
function is very tightly defined by the conditions of the symmetry of
corresponding equation of motion and by the request that the Frenet
curvature be preserved along the extremal curves.

Formula~(\ref{matsyuk:1}) clearly suggests accepting same Lagrange
function also for general Riemannian case,
\begin{equation}\label{matsyuk:2}
L^{R} = k - m{\left\| {{\rm {\bf u}}} \right\|}\;.
\end{equation}
To prove the preservation of curvature~$k$ along
the extremals of~(\ref{matsyuk:2}) we need some further tools  as
introduced below.
\section{Means from higher order mechanics of\\ Ostrohrads\kern-.1em'kyj}
\subsection{Parametric homogeneity}
Let $T^{q}M = {\{ {x^{j},\;u^{j},\;\dot {u}^{j},\;\ddot {u}^{j},\dotsc,
\overset{(q-\1)}u\msp^j}\}}$
denote the manifold of $q$\textsuperscript{th}-order Ehresmann velocities to the
base manifold $M$ of dimension $n$. The prolonged re\-pa\-ra\-me\-t\-rization group
$Gl_{n}^{q} = J_{0}^{q} \left( {\mathbb R,\,\mathbb R} \right)_{0} $ acts on the manifold
$T^{q}M = J_{0}^{q} \left( {\mathbb R,\,M} \right)$ by composition of jets
(in our case $n = 2$). As far as the Lagrange function~(\ref{matsyuk:2}) depends on the
derivatives of at most second order, it lives on the space $T^{2}M$. The infinitesimal counterpart of
the above mentioned parameter transformations of $T^{2}M$ (we put $q=2$) is given by so-called fundamental fields
(for arbitrary order
consult~\cite{matsyuk:b2,matsyuk:b3}):
\begin{equation*}
\zeta _{1} = u^{i}{\frac{{\partial} }{{\partial u^{i}}}} + 2\,\dot
{u}^{i}{\frac{{\partial} }{{\partial \dot {u}^{i}}}}\,,\quad \zeta _{2} =
u^{i}{\frac{{\partial} }{{\partial \dot {u}^{i}}}}\;.
\end{equation*}

If a function $F$ defined on $T^{2}M$ does not change under arbitrary
parameter transformations discussed above, then it with necessity satisfies
the following sufficient conditions:
\begin{equation}\label{matsyuk:3}
\zeta _{1} F = 0\,,\quad \zeta _{2} F = 0\,.
\end{equation}

On the other hand,
if a function $L$ on $T^{2}M$ defines a parameter-independent autonomous
variational problem with the action functional
\[
\int {L{\kern 1pt} \left( {x^{j},\;u^{j},\;\dot {u}^{j}} \right)}
\;d\varsigma \;,
\]
\noindent
then it also with necessity satisfies the so-called Zermelo sufficient
conditions~\cite{matsyuk:b4,matsyuk:b4A}:
\begin{equation}\label{matsyuk:4}
\zeta _{1} L = L\,,\quad \zeta _{2} L = 0\,.
\end{equation}

The generalized momenta are being conventionally introduced by the next
expressions:
\[
p_{i}^{\left( {2} \right)} = {\frac{{\partial L}}{{\partial \dot
{u}^{i}}}}\,,\quad p_{i}^{\left( {1} \right)} = {\frac{{\partial
L}}{{\partial u^{i}}}} - {\frac{{d}}{{d{\kern 1pt} \varsigma
}}}p_{i}^{\left( {2} \right)} \,,
\]
\noindent
while the Hamilton function reads:
\[
H = p_{i}^{\left( {2} \right)} \dot {u}^{i} + p_{i}^{\left( {1} \right)}
u^{i} - L\,.
\]
This Hamilton function may also be expressed in different
way~\cite{matsyuk:b3,matsyuk:b5}:
\begin{equation}\label{matsyuk:5}
H = \zeta _{1} L + {\frac{{d}}{{d{\kern 1pt} \varsigma} }}\zeta _{2} L -
L\,.
\end{equation}
As the Hamilton function is a constant of motion, from
(\ref{matsyuk:3}), (\ref{matsyuk:4}), and
(\ref{matsyuk:5}) we immediately obtain the following
proposition:
\begin{proposition}\label{matsyuk:L II} Let a function $L_{II} $ be parameter-independent, and
let another function $L_{II} $ define a parameter-independent variational
problem on $T^{2}M$. Then $L_{II}$ is constant along the extremals of the
variational problem, defined by the Lagrange function
\begin{equation}\label{matsyuk:6}
L = L_{II} + L_{I} {\rm .}
\end{equation}
\end{proposition}
\noindent This holds because $L_{II} = - H$ with $H$ corresponding to~(\ref{matsyuk:6}).

Frenet curvature is constant along the extremals
of~(\ref{matsyuk:2}), so by the Proposition~\ref{matsyuk:L II} we have
right to state:
\begin{claim}[\cite{matsyuk:b4B,matsyuk:b4C}]
The Lagrange function~(\ref{matsyuk:2}) constitutes the variational principle
for the geodesic circles.
\end{claim}

Now we wish to provide evidence that
in the limit case of Euclidean space the corresponding Euler-Poisson
equation may be specified by means of symmetry considerations together with
the curvature preservation requirement. This means that the inverse
variational problem tools should be applied.

\subsection{The generalized Helmholtz conditions and \\ symmetry.}
Following Tulczyjew (see~\cite{matsyuk:b6,matsyuk:b3}),
let us
introduce some operators, acting in the graded algebra of differential forms who live on
manifolds $T^{q}M$ of varying order $q$ of jets:
\begin{enumerate}
\item The total derivative:
\[
d_{T} f = u^{i}{\frac{{\partial f}}{{\partial x^{i}}}} + \dot
{u}^{i}{\frac{{\partial f}}{{\partial u^{i}}}} + \ddot
{u}^{i}{\frac{{\partial f}}{{\partial \dot {u}^{i}}}}+\dotsb
+
{\overset{\,q} u}\msp^{i}{\frac{{\partial f}}{{\partial\;{\overset{q-\1}u}\msp^{i}}}}\,,
\quad dd_{T} = d_{T}
d\,;
\]
\item
For each of $r\le q$ the derivations of zero degree:
\begin{alignat*}{4}
 i_{0} ( {\omega}  ) &= {\rm d}{\rm e}{\rm g}( {\omega}  )\,\omega \,,
        & i_{r} ( {f} ) &= 0\,,\quad
        && i_{r} ( {dx^{i}} ) = 0\,,& \\
 i_{r} ( {d\,\overset{\,k} u}\msp^{i}  ) &= {\tfrac{{( {k + 1
                        } )\,{\rm !}}}{{( k-r+1 )\,{\rm !}}}}\,d\,\overset{\,k - r}u\msp^{i},\;\quad
        & i_{r} ( {d\,\overset{\,k} u}\msp^{i}  ) &= 0\,,
        &{\rm if}&\quad k < r-1\,;& \\
 \end{alignat*}
\item
The Lagrange derivative:
\[
\delta = \left( {i_{0} - d_{T} i_{1} + {\frac{{1}}{{2}}}\,d_{T}{}^{2}i_{2} -
{\frac{{1}}{{6}}}\,d_{T}{}^{3}i_{3}}+\dotsb+\dfrac{(-1)^q}{q!}\,d_{T}{}^{q}i_{q}  \right)d\,.
\]
\end{enumerate}
It is of common knowledge that the Euler--Poisson expressions constitute
a covariant object.
\begin{lemma}[\cite{matsyuk:b6}] Let a system of some differential expressions of the
third order form a covariant
object---the differential one-form
\begin{equation}\label{matsyuk:7}
\varepsilon = E_{i} \left( {x^{j},\,u^{j},\,\dot {u}^{j},\,\ddot {u}^{j}}
\right)\,dx^{i}.
\end{equation}
Then $\varepsilon = \delta \left( {L} \right)$ for some (local) $L$ if and
only if
\begin{equation}\label{matsyuk:8}
\delta \left( {\varepsilon}  \right) = 0\,.
\end{equation}
\end{lemma}
Developing the criterion~(\ref{matsyuk:8}) amounts to
establishing a general pattern for the
expression~(\ref{matsyuk:7}),
\begin{equation}\label{matsyuk:E}
E_{i} = A_{ij} ( {x^{l},u^{l}} )\,\ddot {u}^{j}
+ {\dot {u}^{p}{\frac{{\partial} }{{\partial {u}^{p}}}}}
        A_{ij} ({x^{l},u^{l}} )\,\dot {u}^{j} + B_{ij} ( {x^{l},u^{l}} )\,\dot {u}^{j}
        + q_{i} ( {x^{l},u^{l}} ),
\end{equation}
\noindent
and to some generalized Helmholtz conditions~\cite{matsyuk:b5,matsyuk:b8A,matsyuk:b8},
 cast in the form of a system of partial
differential equations, imposed on the coefficients $A_{ij} = - A_{ji} $,
$B_{ij} $, and $q_{i} $:
\begin{subequations}
\begin{gather}
 \partial _{u^{{[ {i} }}} A_{{ {jl} ]}}=0 \notag\\
 2\,B{\kern 1pt} _{{\left[ {ij} \right]}} - 3\,\D A_{ij} = 0 \label{matsyuk:a} \\
 2\,\partial _{u^{{\left[ {i} \right.}}} B_{{\left. {j} \right]}\,l}
 -4\,\partial _{x^{{\left[ {i} \right.}}} A_{{\left. {j} \right]\,l}}
 +\partial _{x^{{ {l} }}} A_{{ {ij} }}
 + 2\,\D\partial _{u^{l}} A_{ij} = 0 \label{matsyuk:b}\\
 \partial _{u^{\left( {i} \right.}} q_{\left. {j} \right)} - \D B_{\left(
{ij} \right)} = 0 \notag\\
 2\,\partial _{u^{l}} \partial _{u^{{\left[ {i} \right.}}} q_{{\left. {j}
\right]}}
      -4\,\partial _{x^{{\left[ {i} \right.}}} B_{{\left. {j} \right]\,l}}
+ \D{}^{2}\partial _{u^{l}} A_{ij}
         +6\,\D \partial _{x^{{\left[ {i} \right.}}} A_{{\left. {jl} \right]}}= 0
         \notag\\
 4\,\partial _{x^{{\left[ {i} \right.}}} q_{{\left. {j} \right]}}
 -2\,\D \partial _{u^{{\left[ {i} \right.}}} q_{{\left. {j} \right]}}
 -\D{}^{3}A_{ij} = 0\,, \notag
 \end{gather}
 \end{subequations}
 where the notation $\D=u^p\partial_{x^p}$ was introduced.

The Euclidean symmetry means that everywhere on the submanifold $E$
defined by the system of equations $E_{l} = 0$ the shifted system $X
\left( {E_{l}}  \right)$ vanishes too, where $X$ denotes the prolonged
generator of (pseudo)-orthogonal transformations.
We denote this criterion as
\begin{equation}\label{matsyuk:10}
{\left. X {\left( {E_{l}}  \right)\;} \right|}_{E} = 0\,.
\end{equation}
That we tend to embrace nothing more but only the \textit{geodesic circles} as extremals, falls into
similar condition:
\begin{equation}\label{matsyuk:11}
{\left. {\left(d_{T} k\right)\,} \right|}_{E} = 0\,.
\end{equation}
\noindent As far as in two-dimensional space (${\rm d}{\rm i}{\rm m}\,M = 2$) the
skew-symmetric matrix $A_{ij}$ is invertible, it is not difficult to implement
conditions (\ref{matsyuk:10}) and~(\ref{matsyuk:11}).

If one wishes to include in the set of extremals all those Euclidean geodesics that refer to
the natural parameter, one
has to imply one more condition:
\begin{equation}\label{matsyuk:12}
{\left. {E_{l} \,} \right|}_{{\rm {\bf \dot u=0}}}\,.
\end{equation}

\begin{theorem}\label{matsyuk:unique} Let a third order autonomous dynamical equation
${\rm {\bf E}} = {\rm {\bf 0}}$ in two-dimensional space obey conditions:
\begin{enumerate}
\item
{{\hfil $
\delta ( {\varepsilon}  ) = 0;
$}}
\item
The system of ODEs ${\{ {E_{j} = 0} \}}$ possesses the
Euclidean symmetry;
\item\label{matsyuk:th3}
The system ${\left\{ {E_{j} = 0} \right\}}$ possesses the first
integral --- the Frenet curvature $k$, and includes all curves of constant
curvature as its solutions;
\item\label{matsyuk:th4}
It also includes the strait lines with natural parametrization, ${\rm {\bf
\dot {u}}} = {\rm {\bf 0}}$.
\end{enumerate}
Then
\[
E_{i} = {\frac{{\ep i j \ddot {u}^{j}}}{{{\left\| {{\rm {\bf u}}}
\right\|}^{3}}}} - 3\,{\frac{{\left( {{\rm {\bf \dot {u}}} \cdot {\rm {\bf
u}}} \right)}}{{{\left\| {{\rm {\bf u}}} \right\|}^{5}}}}\;\ep i j \dot {u}^{j} +
m\,{\frac{{{\left\| {{\rm {\bf u}}} \right\|}^{2}\dot {u}_{i} - \left( {{\rm
{\bf \dot {u}}} \cdot {\rm {\bf u}}} \right)u_{i}} }{{{\left\| {{\rm {\bf
u}}} \right\|}^{3}}}}\;.
\]
\end{theorem}
\noindent The Lagrange function is given by~(\ref{matsyuk:1}).
\vskip5\jot
\begin{remarksu}
\hfil\break
\vskip-5\jot
\begin{itemize}
\vskip-1\jot
\item
If, for instance, we took $L=k\sqrt{u_{i} u^{i\mathstrut}}$, then $H = 0$
for this Lagrange function, and the
Proposition~\ref{matsyuk:L II} wouldn't work.
\item Because of the non-degeneracy of the matrix
$A_{ij} $, there cannot exist a parameter-invariant variational problem in
two dimensions that would produce strictly the third order Euler--Poisson equation. But,
if we omit the first addend $k$ in~(\ref{matsyuk:2}), then what
remains defines the conventional parameter-invariant problem for the
Riemannian projective geodesic paths. So, what fixes the parameter along the
extremal in our case, is the Frenet curvature~$k$
in~(\ref{matsyuk:2}).
\end{itemize}
\bigskip
One should confer with
\cite{matsyuk:b9} and~\cite{matsyuk:b10} on these remarks.
\end{remarksu}
\subsection{Proof of the Theorem~\ref{matsyuk:unique}}
Before  passing to the proof of the above Theorem let us notice two simplification
formul{\ae}  which hold at specific occasion of two dimensions. Namely, for arbitrary
vectors $\mathbf a$, $\mathbf c$, $\mathbf v$, and $\mathbf w$ it is true that
\begin{equation}\label{matsyuk:*e}
\nw a c=\dg|\ep i ja^ic^j| \;\;\; \text{and}\;\;\; \nw a c\nw v w=\left|\pw a c\bcdot\pw v w\right|\,,
\end{equation}
where, as usual, $\pw a c\bcdot\pw v w=\pr a v\pr c w-\pr c v\pr a w$~\cite{matsyuk:Dieudonne}.
Also, let us agree to postpone the proof of the second part of statement~\ref{matsyuk:th3}
of Theorem~\ref{matsyuk:unique} until more general Riemannian case proved
in Section~\ref{matsyuk:Completeness}.
\begin{proof}[Proof of the necessity implication of Theorem~\ref{matsyuk:unique} assumptions]
In order to meet the condition~\ref{matsyuk:th4} of the Theorem~\ref{matsyuk:unique}
in the form~(\ref{matsyuk:12}), we have to remove the array $\mathbf q$ from~(\ref{matsyuk:E}).
Next we write down the first part of the statement~\ref{matsyuk:th3} given
by means of~(\ref{matsyuk:11}). Starting with the expression
\begin{equation}\label{matsyuk:k}
k=\dfrac{\nw u{\dot u}}{\pnu 3}
\end{equation}
of the Frenet curvature we substitute~$\ddu$ in
\[
d_Tk=\dfrac{\pw u{\du}\bcdot\pw u{\ddu}}{\pnu 3\nw u{\du}}
-3\,\dfrac{\nw u{\du}\pr u{\du}}{\pnu5}
\]
by $\ddu=-A^{-1}(\du.\boldsymbol{\partial}_{u})\,A\du-A^{-1}B\du$ of~(\ref{matsyuk:E})
and then split the expression~(\ref{matsyuk:11}) by the powers of~$\du$ to
obtain separately
\begin{subequations}
\begin{gather}\label{matsyuk:u3}
 \pr u u \pw u{\du}\;\bcdot\left(\bu\wedge\left(A^{-1}(\du.\boldsymbol\partial_u)A\du\right)\right)
 +3\,\pr u{\du}\nw u{\du}^2=0\\
\label{matsyuk:u2} \pw u{\du}\;\bcdot\left(\bu\wedge\left(A^{-1}B\du\right)\right)=0\,.
\end{gather}
\end{subequations}
Let us recall that the covariant and the contravariant Levi--Civita symbols are
related by $\ep i je^{jl}=-\delta_i{}^l$ and also let matrix~$A$ be expressed
as $A_{ij}=A_{12}\ep i j$. With these agreements the first addend in~(\ref{matsyuk:u3})
becomes
\[
\dfrac{1}{A_{12}}\pnu2\nw u{\du}^2(\du.\boldsymbol\partial_u)A_{12}\,,
\]
thus reducing~(\ref{matsyuk:u3}) by means of~(\ref{matsyuk:*e}) to the partial differential equation
\[
\pnu2(\du.\boldsymbol\partial_u)A_{12}+3\,A_{12}\pr u\du=0
\]
that in turn yields the solution
\[
A_{12}=\alpha\pnu{-3}.
\]
Now we see that matrix~$A$ satisfies the relations
\begin{equation}\label{matsyuk:dA=3A}
\du.\bp_u\,A=-3\,\dfrac{\bu\bcdot\du\;}{\pnu2}\,A\,,
\end{equation}
and, evidently,
\begin{equation}\label{matsyuk:rotA=0}
 e^{ij}u_i\frac{\partial }{\partial u^j}\,A=0\,,
\end{equation}
 with the help of which the Euler--Poisson expression~(\ref{matsyuk:E}) becomes
\begin{equation}\label{matsyuk:new E}
 \mathbf E=A\ddu-3\,\dfrac{\bu\bcdot\du\;}{\pnu2}A\du+B\du\,,
\end{equation}
 so that the submanifold $\mathbf{E=0}$ is now defined by the equation
\begin{equation}\label{matsyuk:E=0}
\ddu=3\,\dfrac{\bu\bcdot\du\;}{\pnu2}A^{-1}B\du\,.
\end{equation}

Again with the help of~(\ref{matsyuk:*e}) the equation~(\ref{matsyuk:u2})
takes the shape
\[
\left\|\bu\wedge \left(A^{-1}B\du\right)\right\|=0,\quad\text{or}\quad\ep i je^{jp}B_{pl}u^i\dot u^l=0\,,
\]
{from} where it follows that
\begin{equation}\label{matsyuk:uB}
u^pB_{pl}=0.
\end{equation}

The generators of the Euclidean transformations are enumerated by an arbitrary constant $\varpi$
and an arbitrary constant array~$\bx=\{\chi^i\}$ and they read:
\begin{subequations}
\begin{gather}\label{matsyuk:chi}
        \bx.\bp_x\left(\equiv\chi^i\frac{\partial }{\partial x^i}\right)\,;\\
        \label{matsyuk:varpi}
\varpi e^{ij}\left(x_i\dfrac{\partial}{\partial x^j}
        +u_i\dfrac{\partial}{\partial u^j}
        +\dot u_i\dfrac{\partial}{\partial \dot u^j}
        +\ddot u_i\dfrac{\partial}{\partial \ddot u^j}\right).
\end{gather}
\end{subequations}

Applying criterion~(\ref{matsyuk:10}) with $X=\bx.\bp_x$ and taking into account
the substitution~(\ref{matsyuk:E=0}) ends in
\begin{equation}\label{matsyuk:sym4}
-\dfrac{\bx.\bp_x\,\alpha}{\alpha}\,B\du+\bx.\bp_x\,B\du=\mathbf0.
\end{equation}

Applying criterion~(\ref{matsyuk:10}) with $X$ equal to~(\ref{matsyuk:varpi}) and again calling to mind
the substitution~(\ref{matsyuk:E=0}) with the help of
\[
A_{lj}e^{ij}A^{-1}{}_i{}^p=\frac 1{A_{12}}A^{-1}{}_l{}^p=-g_{il}e^{ip}
\]
ends in
\[
g_{ij}e^{il}B_{lp}\dot u^p+e^{il}u_i\frac{\partial }{\partial u^l}\,B_{jp}\dot u^p
+e^{il}B_{jl}\dot u_i=0\,,\quad\text{identically with respect to $\dot u^p$},
\]
{from} where we conclude:
\begin{equation}\label{matsyuk:sym*}
e^{il}u_i\frac{\partial }{\partial u^l}\,B_{jp}+g_{ij}e^{il}B_{lp}
+g_{ip}e^{il}B_{jl}=0\,.
\end{equation}
We may deduce {from}~(\ref{matsyuk:sym*}) that the skew-symmetric part
of~$B$ should satisfy the equation:
\begin{equation}\label{matsyuk:sym**}
e^{ij}u_i\,\frac{\partial }{\partial u^j}\,B_{[lp]}
+g_{il}e^{ij}B_{[jp]}+g_{ip}e^{ij}B_{[lj]}=0.
\end{equation}
Let the skew-symmetric part of matrix~$B$ be presented as~$\beta\ep i j$.
Then equation~(\ref{matsyuk:sym**}) confirms that~$\beta$ should be
a differential invariant:
\begin{equation}\label{matsyuk:sym beta}
e^{ij}u_i\,\frac{\partial }{\partial u_j}\,\beta=0.
\end{equation}
But the variationality condition~(\ref{matsyuk:a}) now says:
\begin{equation}\label{matsyuk:sym a}
2\beta=3\,\bu.\bp_x\alpha\,.
\end{equation}
Applying the left hand side operator of~(\ref{matsyuk:sym beta}) to~(\ref{matsyuk:sym a})
along with equation~(\ref{matsyuk:rotA=0}) produces
\[
\ep j ie^{ip}\,\frac{\partial }{\partial x^p}\,\alpha=0.
\]
Thus $\alpha$ does not depend on~$x^{i}$. Looking back at~(\ref{matsyuk:sym a}) immediately implies $\beta=0$,
 matrix~$B$ being symmetric thus.
 In addition, we see that matrix~$B$ also
should not depend on~$x^i$ by the reason of relation~(\ref{matsyuk:sym4}).

Now it is time to turn back to the constraint~(\ref{matsyuk:uB}).
Of course, we could have used it much
earlier, but we prefer to unleash it now. So, the two equations, contained
there, allow us to prescribe the shape to the matrix~$B$ as follows (independent
of its virtual symmetry). Let $B_{12}=b_1u_2$, $B_{21}=b_2u_1$. Then {from}~(\ref{matsyuk:uB})
one has:
\[
B_{ij}=b_iu_j-(\mathbf b\bcdot\bu)g_{ij}\,.
\]
But we already know that $B_{[ij]}=0$. This immediately implies that $\mathbf b$
and $\bu$ must be collinear,
$\mathbf b=\mu\,\bu$, thus suggesting the following form of matrix~$B$:
\begin{equation}\label{matsyuk:B}
B_{ij}=\mu\,\left(u_iu_j-\pr u u\,g_{ij}\right)\,.
\end{equation}
Let us again act on~(\ref{matsyuk:B}) with the operator $e^{ij}u_i\frac{\partial }{\partial u^j}$
and make use of~(\ref{matsyuk:sym*}). After some simplifications we get:
\[
e^{ij}u_i\frac{\partial }{\partial u^j}\,\mu=0\,,
\]
what suggests that~$\mu$ depends on $u^i$ exclusively via~$\bu\bcdot\bu$.

The definite step consists in applying the second valid variational criterion,
that of~(\ref{matsyuk:b}). It is efficient to make contraction of~(\ref{matsyuk:b})
with $u^i$ on the left and in meanwhile not to forget about the constraint~(\ref{matsyuk:uB}).
One obtains:
\[
u^i\frac{\partial }{\partial u^i}\,B_{jp}=-B_{jp}\,.
\]
Together with the guise~(\ref{matsyuk:B}) this produces
\[
\left(2\,\pr u u\,\frac{\partial \mu}{\partial \pnu2}+3\,\mu\right)
\left(u_iu_j-\pr u ug_{ij}\right)=0\,,
\]
what clearly has the solution $\mu=\frac{m}{\pr u u^{3/2}}$ and so says the finite appearance
of~$B$:
\[
B_{ij}=\dfrac{m}{\pr u u^{3/2}}\,\left(u_iu_j-\pr u u\,g_{ij}\right)\,.
\]
\end{proof}

\section{The variational description of geodesic circles}
\subsection{The variational equation}
Before calculating the variation of the integrand in the functional expression $\int k\,d\varsigma$ let us
agree on some basic formul\ae. If~$\upsilon$ denotes the infinitesimal shift
of the path $x^i(\varsigma)$ and if~$\T$ stands for the covariant differentiation operator
according to that shift,
then the covariant variation of any vector field $\xi$ along this path is given
by
\begin{equation}\label{matsyuk:Cov}
\langle\upsilon,\T\xi\rangle\msp^i=\langle\upsilon,d\xi^i\rangle+\Gamma^i_{lj}\msp\xi^j\upsilon^l.
\end{equation}
Let the covariant derivative of a vector field be notated by prime. And let us
introduce a special designation for the evaluation of Riemannian
curvature on velocities as follows:
\begin{equation*}
\sigma^l{}_j=R_{ji,p}{}^lu^iu^p.
\end{equation*}
The vector differential one-form $\bs=[\sigma^l{}_j]$ is semi-basic when
the projection $TM\to M$ is considered:
$\langle\upsilon,\bs\rangle\msp^l=\sigma^l{}_j\upsilon^j$. Let~$\bt$ denote the
vector one-form representing the identity: $\bt=[\delta^l{}_j]$.
Next formul{\ae}  replace then the usual interchange rule between infinitesimal
variation and ordinary differentiation:
\begin{equation}\label{matsyuk:sigma}
\begin{split}
\mathbf\T\bu&=\bt\msp\bpr\;\;\text{[this recapitulates definition~(\ref{matsyuk:Cov})]},\\
\mathbf\T(\bw)&=(\mathbf\T\bu)\msp\bpr-\bs\;\;\text{[this recapitulates
the definition of the tensor $R_{ji,p}{}^l$
]}.
\end{split}
\end{equation}

Further on we shall find escape from highly tangled and tedious calculations in the truth of the following
relation (valid in two dimensions only):
\begin{equation}\label{matsyuk:***e}
\pr a a\,\pw v c\bcdot\pw v c-\pr a v\,\pw v c\bcdot\pw a c+\pr a c\,\pw v c\bcdot\pw a v=0\,,
\end{equation}
along with the simplification formul\ae~(\ref{matsyuk:*e}).

The above formal and highly symbolic notations
 save place and time and help to avoid unessential calculative details, whereas
keeping the skeleton of the variational procedure untouched and
faithfully tracing the logical outlines of our development as well as producing the correct final result.

With these prerequisites we calculate the covariant variation of the Frenet
curvature~(\ref{matsyuk:k}), discarding terms which present total
covariant derivatives:
\begin{align*}
\T k&=\dfrac{\pw u\bw\bcdot\pw{\T u}\bw}{\pnu3\nw u\bw}
        -3\,\dfrac{\nw u\bw}{\pnu5}\,\pr u{\T u}
        +\dfrac{\pw u\bw\bcdot\pw u{\T\bw}}{\pnu3\nw u\bw} = \\
  \intertext{\endgraf [by (\ref{matsyuk:*e}), (\ref{matsyuk:sigma}), and Leibniz rule]}
\begin{split}
      = 2\,\dfrac{\nw{\T u}\bw}{\pnu3}-3\,\dfrac{\nw u\bw}{\pnu5}\,\pr u{\T u}
     -3\,\dfrac{\nw{\T u}u}{\pnu5}\,\pr u\bw\\
     -\dfrac{\pw u\bw\bcdot\pw u\bs}{\pnu3\nw u\bw}
\end{split}\\
&\begin{alignedat}{2}
     &=-\dfrac{\nw{\T u}\bw}{\pnu3}
     -\dfrac{\pw u\bw\bcdot\pw u\bs}{\pnu3\nw u\bw}
        &\;\;&
          \text{[by~(\ref{matsyuk:***e})]} \\
      &=\dfrac{\nw\bt{\bw\bpr}}{\pnu3}
      -3\,\dfrac{\nw\bt{\bw}}{\pnu5}\,\pr u\bw
     -\dfrac{\pw u\bw\bcdot\pw u\bs}{\pnu3\nw u\bw}
        &&
        \text{[by Leibniz rule again].}
\end{alignedat}
\end{align*}
Let us introduce one more succinct notation:
\[
\boxed{\mathcal R_j=\dfrac{\dg}{\pnu3}\,\ep i l R_{jn,p}{}^lu^iu^pu^n}
\]
The relation between this scalar semi-basic one-form $\mathcal R_jdx^j$ and
previously introduced vector semi-basic one form $\sigma^i{}_jdx^j$ is obvious:
\[
\dg\dfrac{\ep i lu^i\sigma^l{}_j}{\pnu3}=\mathcal R_j.
\]
Both quantities satisfy the constraint imposed on the contraction with velocity:
\begin{equation}\label{matsyuk:uR=0}
\mathcal R_ju^j=0\,,
\end{equation}
along with
\begin{equation}\label{matsyuk:usigma=0}
u_i\sigma^i{}_j=0\,.
\end{equation}
Now the Euler--Poisson equation for the complete Lagrange function~(\ref{matsyuk:2})
may be expressed in the form, valid in each case of different signature
of metric tensor~$g_{ij}$ with the help of Hodge star operator:
\begin{equation}\label{matsyuk:Myeq}
\boxed{\mathbf E^R=-\dfrac{*\,\bw\bpr}{\pnu3}+3\,\dfrac{\pr u\bw}{\pnu5}\,*\bw
        +m\,\dfrac{\pr u u\,\bw-\pr\bw u\,\bu}{\pnu3}-\R=\boldsymbol0}
\end{equation}
\begin{remarku}
The force~$\R$ may be given another shape thanks to the relation~(\ref{matsyuk:usigma=0}):
\[
\mathcal R_ldx^l=\dfrac{\pw u\bw\bcdot\pw u\bs}{\pnu3\nw u\bw}=\dfrac{\bs\bcdot\bw}{\bnu\nw u\bw}
=\dfrac{1}{2}\,R_{lj,pi}u^jS^{pi}dx^l\,,
\]
where $S^{pi}=\dfrac{\pw u\bw\msp^{pi}}{\bnu\nw u\bw}$ is a formally introduced
`spin' tensor.
\end{remarku}
\subsection{\label{matsyuk:Completeness}Completeness of variational description of geodesic circles}
It remains to prove that every geodesic circle may be given a consistent
pa\-ra\-me\-tri\-zation, which makes it an extremal of the variational problem with the Lagrange
function~(\ref{matsyuk:2}).

\smallskip
\noindent\textit{The governing equation for the geodesic circles.}
With the intention to derive a dynamical differential equation, governing the
motion along a geodesic path,
we put equal to zero the derivative of the Frenet curvature function~$k$
in terms of natural parametrization by $ds=\sqrt{u_iu^{i\mathstrut}}\,d\varsigma$:
\begin{equation}\label{matsyuk:k'=0}
\bw_s\bcdot\bw\bpr_s=0.
\end{equation}
To it we add the obvious constraint
\begin{equation}\label{matsyuk:u'u'=0}
\bw_s\bcdot\bw_s+\bu_s\bcdot\bw\bpr_s=0,
\end{equation}
which merely presents the differential consequence of
\begin{equation}\label{matsyuk:uu'=0}
\bu_s\bcdot\bw_s=0.
\end{equation}
Next we solve  the system of equations (\ref{matsyuk:k'=0}) and (\ref{matsyuk:u'u'=0})
for~$\bw\bpr_s$ to obtain
\begin{equation}\label{matsyuk:u''1}
(u''_s)_l=\frac{\ep l i(u'_s)^i}{\ep i j(u'_s)^i(u_s)^j}\,\bw_s\bcdot\bw_s\,.
\end{equation}
We leave it to the Reader to check with the help of~(\ref{matsyuk:uu'=0}) and
 of $\bu_s\bcdot\bu_s=1$, that in two--dimensional space 
 the relation $\ep i l(u'_s)^i=(u_s)_l\ep i j(u'_s)^i(u_s)^j$ 
 holds that reduces 
  equation~(\ref{matsyuk:u''1})
 to the well known governing equation of geodesic circles
\begin{equation}\label{matsyuk:u''2}
 \bw\bpr_s+(\bw_s\bcdot\bw_s)\,\bu_s=0.
\end{equation}
In order to dispense with the constraint $\bu_s\bcdot\bu_s=1$ we recalculate
the derivatives in~(\ref{matsyuk:u''2}) by the reparametrization from~$s$
to an arbitrary elapse parameter~$\varsigma$ along the path of
a geodesic circle to
see at last  that
geodesic circles accept characterization as the integral curves of
the following parameter-homogeneous differential equation:
\begin{equation}\label{matsyuk:geocircles}
\dfrac{\bw\bpr}{\pnu3}=\dfrac{\bu\bcdot\bw\bpr}{\pnu5}\,\bu
        +3\,\dfrac{\bu\bcdot\bw}{\pnu5}\,\bw
        -3\,\dfrac{\pr u\bw^2}{\pnu7}\,\bu\,.
\end{equation}

\smallskip
\begin{proof}[Proof of the exhaustivenes of extremal set]
Let us complement equation~(\ref{matsyuk:geocircles}) by the following additional one, which is consistent
with the equation~(\ref{matsyuk:Myeq}) (as its consequence) and will play the role of the
 means to fix the way of parametrization along the extremal curve:
\begin{equation}\label{matsyuk:Parameter}
\dfrac{\bu\bcdot\bw\bpr}{\pnu3}
        -3\,\dfrac{\pr u\bw\msp^2}{\pnu5}
        =*\,\left(\dfrac{m}{\bnu}\;\bu\wedge\bw-\bu\wedge\R\right)\,.
\end{equation}
For the sake of efficiency, let us evaluate the Euler--Poisson expression~(\ref{matsyuk:Myeq})
on some arbitrary vector~$\by$:
\[
\mathbf E^R.\,\by=\dfrac{*\;\pw\by{\bw\bpr}}{\pnu3}
-3\,\dfrac{\bu\bcdot\bw}{\pnu5}\;\:*\;\pw\by\bw
+\dfrac{m}{\pnu3}\,\pw u\bw\bcdot\pw u\by-\R\,.\,\by\,.
\]
If now we substitute~$\bw\bpr$ in this equation with the expression {from}~(\ref{matsyuk:geocircles})
and simultaneously take into account the additional equation~(\ref{matsyuk:Parameter}),
we will get:
\begin{align*}
\begin{split}
\mathbf E^R.\,\by=-\dfrac{*\;\pw\by u\;*\;\pw u\R}{\pnu2}
&+\dfrac{m}{\pnu3}\;*\;\pw\by u\;*\;\pw u\bw\\
&+\dfrac{m}{\pnu3}\,\pw u\bw\bcdot\pw u\by-\R\,.\,\by\\
\end{split}\\
&\phantom{\mathbf E^R.\,\by}=-\dfrac{\pr\by u\,\pr u\R}{\pnu2}+\by\bcdot\R-\R\,.\,\by\equiv0
\end{align*}
because of~(\ref{matsyuk:uR=0})
\end{proof}
\subsection*{Acknowledgment}My participation in this excellent conference was made possible
exclusively due to the generosity of the organizes to whom my deepest gratitude
should and actually is being readily expressed here.

\noindent
       Roman Matsyuk \\
       Institute for Applied Problems in Mechanics and Mathematics, L\kern-.1em'viv, Ukraine\\
       15 Dudayev St., 79005 L\kern-.1em'viv, Ukraine\\
        matsyuk@lms.lviv.ua, romko.b.m@gmail.com\\
        \url{http://iapmm.lviv.ua/12/eng/files/st_files/matsyuk.htm}
\end{document}